\documentclass [11pt,oneside]{amsart}
\usepackage[top=1in, bottom=1in, left=1.2in, right=1.2in]{geometry}
\usepackage{tikz}
\usepackage[active]{srcltx}
\usepackage[curve]{xypic}
\usepackage{graphicx}
\usepackage{mathrsfs}

\usepackage{color}

\usepackage{caption}
\usepackage{subcaption}

\usepackage{xcolor}
\usepackage{amssymb,amsmath,amsthm,amsfonts}
\usepackage{enumerate}
\usepackage{setspace}
\usepackage{empheq}
\usepackage[all]{xy}
\usepackage{verbatim}
\usepackage[normalem]{ulem}
\usepackage{todonotes}
\usepackage[bookmarks,colorlinks,breaklinks]{hyperref}  
\hypersetup{linkcolor=blue,citecolor=blue,filecolor=blue,urlcolor=blue} 

\numberwithin{equation}{section}
\numberwithin{figure}{section}

\definecolor{myblue}{rgb}{0.6, 0.9, 1}

\makeatletter

\newcommand{\Rmnum}[1]{\expandafter\@slowromancap\romannumeral #1@}
\makeatother

\definecolor{myblue}{rgb}{0.6, 0.9, 1}
\definecolor{mygreen}{rgb}{0,0,1}
\definecolor{purple}{rgb}{0.6,0.2,1}
\definecolor{orange}{rgb}{0.8,0,0.2}

\newcommand{\Z}{\mathbb{Z}}

\theoremstyle{plain}

\newtheorem{theorem}{Theorem}[section]
\newtheorem{proposition}[theorem]{Proposition}
\newtheorem{lemma}[theorem]{Lemma}
\newtheorem{corollary}[theorem]{Corollary}

\theoremstyle{definition}
\newtheorem{definition}{Definition}[section]

\theoremstyle{remark}
\newtheorem{remark}[theorem]{Remark}
\newtheorem{assumption}[definition]{Assumption}

\def\Z{\mathbb{Z}}

\def\Z2{\mathbb{Z}_{2}}
\def\m2#1{\ ({\rm mod} \ 2^{#1})}

\makeatletter
\@namedef{subjclassname@2020}{%
  \textup{2020} Mathematics Subject Classification}
\makeatother

\begin{document}

\title[The $p$-adic subhyperbolic rational maps]{Expanding property and  statistical laws for $p$-adic subhyperbolic rational maps} 

\author{Shilei Fan}
\address{School of Mathematics and Statistics, and Key Laboratory of Nonlinear Analysis \& Applications (Ministry of Education), Central China Normal University, Wuhan 430079, P. R. China}  \email{slfan@mail.ccnu.edu.cn}

\author{Lingmin Liao}
\address{School of Mathematics and Statistics, Wuhan University, Wuhan 430072, Hubei, China}
\email{lmliao@whu.edu.cn}

\author{Hongming Nie}
\address{Institute for Mathematical Sciences, Stony Brook University, Stony Brook, NY 11794-3660, USA}
\email{hongming.nie@stonybrook.edu}

\author{Yuefei Wang}
\address{College of Mathematics and Statistics, Shenzhen University, Shenzhen 518060, Guangdong, China \& Academy of Mathematics and System Sciences, CAS, Beijing 100190, China}
\email{wangyf@math.ac.cn}

\thanks{S. L. FAN was partially supported by NSFC (grants No.  12331004 and No. 11971190) and Fok Ying-Tong Education Foundation, China (grant No.171001).  Y. F. WANG was partially supported by NSFC (grants No. 12231013) and NCAMS .}

\begin{abstract}
 Let $K$ be a finite extension of the field $\mathbb{Q}_p$ of $p$-adic numbers. A rational map $\phi\in K(z)$ of degree at least $2$ is  subhyperbolic if each critical point in the $\mathbb{C}_p$-Julia set of $\phi$ is eventually periodic.  We show that subhyperbolic maps in $K(z)$ exhibit expanding property
with respect to  some (singular) metric. As an application, under a mild assumption, we establish several statistical laws for such maps in $K(z)$ with compact $\mathbb{C}_p$-Julia sets. 
 \end{abstract}

\maketitle


\section{Introduction}\label{sec:intro}

The uniform expanding maps in smooth dynamical systems exhibit interesting ergodic and statistical properties. The hyperbolic maps in the dynamics of rational maps over number fields are  paradigms of uniform expanding maps, see  \cite{Benedetto01} and \cite[Section 19]{Milnor06}. Beyond hyperbolicity, in complex dynamics the subhyperbolic maps already appear in the literature and have been studied in a wide sense. Such maps induce natural orbifolds and  possess uniform expanding property with respect to some singular metric. This singular metric can be deducted from the metric in the universal covering space of the corresponding orbifolds, see \cite[Section 19]{Milnor06} for details. In the context of non-archimedean dynamics, we defined subhyperbolic  maps and explored certain dynamical properties in  our previous work \cite{Fan21}. In this present paper, we establish the non-archimedean counterpart of  expanding property for non-archimedean subhyperbolic maps and then explore the corresponding statistical laws. Due to the absence of tools (e.g. orbiford,  universal covering space) employed in the complex setting, we construct a desired (singular) metric relying on non-archimedean analysis.

Throughout this paper, we let $K$ be a finite extension of the field $\mathbb{Q}_p$ of $p$-adic numbers associated with the absolute value restricted from the natural and nontrivial non-archimedean absolute value $|\cdot|$ on $\mathbb{C}_p$. For any extension $L$ of $K$ and any rational map $\phi\in K(z)$ of degree at least $2$, we denote by $F_L(\phi)$ and $J_L(\phi)$ the Fatou set and Julia set, respectively, of $\phi$ in the projective space $\mathbb{P}^1_L$, defined by the terminology of equicontinuity. We assume that (by a change of coordinate if necessary) the point $\infty$ is contained in $F_{\mathbb{C}_p}(\phi)$. We remark here that  this assumption always holds if we regard $\phi\in K(z)$ as a rational map defined over some finite extension of $K$, due to the existence of non-repelling fixed points, see \cite[Proposition 4.2 and Theorem 5.14]{Benedetto19}. Also note that if $L$ is a finite extension of $K$, then $\mathbb{P}^1_L$ is compact, so is the intersection $J_{\mathbb{C}_p}(\phi)\cap\mathbb{P}^1_L=J_{\mathbb{C}_p}(\phi)\cap L$; moreover, the field $L$ carries  a normalized Haar measure $\nu_L$, see \cite{Haar33, Weil41}.

 Following our previous work \cite[Definition 1.2]{Fan21}, a rational map $\phi\in K(z)$ of degree at least $2$ is said to be  \emph{subhyperbolic} if each critical point in $J_{\mathbb{C}_p}(\phi)$ is eventually periodic. 
 To state the expanding property for subhyperbolic rational maps in $K(z)$,  we use the following definition of admissible functions. It is an analog of the admissibility in complex dynamics (c.f. \cite[Section V.3]{Carleson93})

\begin{definition}\label{def:admissible}
Let $\phi\in K(z)$ be a rational map of degree at least $2$ with $\infty\in F_{\mathbb{C}_p}(\phi)$, and let $L$ be a finite extension of $K$. For a continuous (extended) real-valued function $\rho:=\rho_L$ defined on some neighborhood  $V\subset\mathbb{C}_p$ of $J_{\mathbb{C}_p}(\phi)\cap L$, we say that $\rho$ is \emph{admissible} if the following hold:
\begin{enumerate}
\item there exists $C>0$ such that $\rho(V)\subset [C,+\infty]$;
\item the function $\rho$ (possibly) blows up at finitely many exceptional points $z_1,\cdots,z_\ell$ in $J_{\mathbb{C}_p}(\phi)\cap L$; and
\item  there exist $B>0$ and $0<\beta<1$ such that $\rho(z)\le\sum_{i=1}^\ell B/|z-z_j|^\beta$. 
\end{enumerate}
\end{definition}


As we will see in Proposition \ref{prop:exp}, due to the discreteness of $|L^\times|$, an admissible function $\rho_L$ induces a metric $\rho_L(z)d\nu_L(z)$ on $V\cap L$, called the \emph{admissible metric} induced by $\rho_L$. This will allow us to work on the diameters of sets in $L$, see Section \ref{sec:exp}.


\begin{definition}
Let $\phi\in K(z)$ be a rational map of degree at least $2$ with $\infty\in F_{\mathbb{C}_p}(\phi)$, and let $L$ be a finite extension of $K$. We say that $\phi$ is \emph{expanding} in $J_{\mathbb{C}_p}(\phi)\cap L$ with respect to an admissible metric $\rho_Ld\nu_L$ if there exists $\lambda>1$ such that for any $z\in J_{L}(\phi)$ whenever $z$ and $\phi(z)$ are not exceptional points of $\rho_L$,
\begin{equation}\label{equ:expansion}
\rho_L(\phi(z))|\phi'(z)|\ge\lambda\rho_L(z).
\end{equation}
\end{definition}

We now can state the equivalence between subhyperbolicity and expanding  property . 

\begin{theorem}\label{thm:equivalence}
Let $\phi\in K(z)$ be a rational map of degree at least $2$ such that the point $\infty$ is contained in $F_{\mathbb{C}_p}(\phi)$.
Then the following are equivalent:
\begin{enumerate}[{\rm (1)}]
\item The map $\phi$ is subhyperbolic, i.e. each critical point in $J_{\mathbb{C}_p}(\phi)$ is eventually periodic.
\item For any finite extension $L$ of $K$, the map $\phi$ is expanding in $J_{\mathbb{C}_p}(\phi)\cap L$ with respect to an admissible metric.
\end{enumerate}
\end{theorem}

Let us exploit the admissible metric in Theorem \ref{thm:equivalence} (2). For a subhyperbolic rational map $\phi\in K(z)$ as in Theorem \ref{thm:equivalence}, denote by $P(\phi)\subset\mathbb{C}_p$ the postcritical set of $\phi$. We assign each point $z_0\in P(\phi)\cap J_{\mathbb{C}_p}(\phi)$ an index $\Xi(z_0)\in(0,1)$ sufficiently close to $1$ such that $1-\Xi(z_0)=(\deg_{z_0}\phi)\cdot(1-\Xi(\phi(z_0)))$ if $\phi(z_0)\not=z_0$.  We also assign such a point $z_0$ a suitably chosen number $B(z_0)>0$. Then we set $\theta(z)=1$ if $z$ is not contained in a  small prescribed neighborhood of $P(\phi)\cap J_{\mathbb{C}_p}(\phi)$; and if  otherwise, set
$$\theta(z):=\frac{B(\iota(z))}{|z-\iota(z)|^{\Xi(\iota(z))}},$$
where $\iota(z)\in P(\phi)$ is the point such that the distance $\mathrm{dist}(z,P(\phi)\cap J_{\mathbb{C}_p}(\phi))$ under the metric induced by the absolute value $|\cdot|$ on $\mathbb{C}_p$ is equal to $|z-\iota(z)|$. As we will see in Section \ref{sec:index}, the factor $B(\iota(z))$ guarantees the expanding property of $\phi$ near $P(\phi)\cap J_{\mathbb{C}_p}(\phi)$ with respect to $\theta$.
Finally, due to the compactness of $J_{\mathbb{C}_p}(\phi)\cap L$, we can find a large integer $N>0$ and define a desired admissible function by
$$\rho_L(z)=\left(\prod_{j=0}^{N-1}\theta(\phi^j(z))\right)^{1/N}\cdot\prod_{j=0}^{N-2}|\phi'(\phi^j(z))|^{(N-1-j)/N}.$$
The above function $\rho_L$ induces a desired admissible metric, see Section \ref{sec:expansion}. Moreover, the exceptional points  of $\rho_L$ are the points in $P(\phi)\cap J_{\mathbb{C}_p}(\phi)\cap L$. It is worth mentioning that if $\phi$ is hyperbolic, that is, the Julia set $J_{\mathbb{C}_p}(\phi)$ contains no critical points, then the corresponding admissible function has no exceptional points, also see \cite[Section 3]{Benedetto01}.


Although the constructions of admissible metrics for sub-hyperbolic maps in both the complex setting and our setting are focus on the postcritical sets, there are differences in these two settings. In the complex setting, as aforementioned, the metric relates to orbifold; while in our setting, we directly work on  a neighborhood of the Julia postcritical set inspired by our previous work \cite[Theorem 1.4]{Fan21} on the description the dynamics of such maps via symbolic dynamics. We obtain the expanding property  by computing the derivatives with the non-archimedean property.



\smallskip

As an application of Theorem \ref{thm:equivalence}, we will establish some statistical laws for subhyperbolic maps in $K(z)$,
under the assumption that the $\mathbb{C}_p$-Julia set is compact and contains a dense subset consisting of algebraic elements. This assumption guarantees that the $\mathbb{C}_p$-Julia set carries  a so-called exponentially contracting metric, see Definition \ref{def:con} and Proposition \ref{prop:con}.


To formulate the result, we will use the following notation. Let $\psi\in\mathbb{C}_p(z)$ be a rational map of degree at least $2$ having non-empty  Julia set $J_{\mathbb{C}_p}(\psi)$. For any continuous function $f: J_{\mathbb{C}_p}(\psi)\to\mathbb{R}$ and for any $\psi$-invariant probability measure $\mu$ on $J_{\mathbb{C}_p}(\psi)$,
we set
$$S_nf(z):=\sum_{k=0}^{n-1}f(\psi^k(z)),$$
and define (if exists)
$$\sigma_{\mu}(f)^2:=\lim_{n\to\infty}\frac{1}{n}\int_{J_{\mathbb{C}_p}(\psi)}\left(S_nf(z)-n\int_{J(_{\mathbb{C}_p}\psi)}f d\mu\right)^2d\mu.$$
Moreover, in the case that $J_{\mathbb{C}_p}(\psi)$ is compact and $f$ is H\"older continuous, we can define the topological pressure $\mathcal{P}_{top}(f)$ and consider corresponding equilibrium state $\mu_f$, see Section \ref{sec:proof2}.


\begin{theorem}\label{thm:CLT}
Let $\phi\in K(z)$ be a subhyperbolic rational map of degree at least $2$.  Assume that $J:=J_{\mathbb{C}_p}(\phi)$ is compact and $J_{\overline{\mathbb{Q}}_p}(\phi)$ is dense in $J$. Let $\tilde\rho_\infty$ be an exponentially contracting metric on $J$ compatible with the topology. Consider a H\"older continuous function $f: (J, \tilde\rho_\infty)\to\mathbb{R}$. Then there exists a unique equilibrium state $\mu_f$ for $f$ on $J$. Moreover,  for any H\"older continuous function $g: (J, \tilde\rho_\infty)\to\mathbb{R}$, the limit $\sigma:=\sigma_{\mu_f}(g)\in[0,\infty)$ exists and the following statistical laws  hold.

\begin{enumerate}[{\rm (1)}]
\item {\rm (Central Limit Theorem)}  If $\sigma>0$, then for any $a<b$, as $n\to\infty$,
$$\mu_f\left(\left\{z\in J:\frac{S_ng(z)-n\int_{J}g d\mu_f}{\sqrt{n}}\in[a,b]\right\}\right)\rightarrow\frac{1}{\sqrt{2\pi\sigma^2}}\int_a^b e^{-t^2/2\sigma^2}dt;$$
 and if otherwise $\sigma=0$, one has convergence in probability to the Dirac mass at $0$.

 \item {\rm (Law of Iterated Logarithm)} For $\mu_f$-a.e. $z\in J$,
 $$\limsup_{n\to\infty}\frac{S_ng(z)-n\int_{J}g d\mu_f}{\sqrt{n\log n\log n}}=\sqrt{2\sigma^2}.$$

 \item {\rm (Exponential Decay of Correlations)} There exist constants $\beta>0$ and $C\ge 0$, independent of $g$, such that for any $\mu_f$-integrable function $\chi: J\to\mathbb{R}$ and for any $n\ge 0$,
 $$\left|\int_Jg\cdot(\chi\circ\phi^n)d\mu_f-\int_Jg d\mu_f\cdot\int_J\chi d\mu_f\right|\le Ce^{-n\beta}||\underline{\chi}||_1\cdot||\underline  g||_\alpha,$$
 where,  $\underline{\chi}=\chi-\int_J\chi d\mu_f$, $\underline{g}=g-\int_Jg d\mu_f$, and $\alpha$ is the H\"older exponent of $g$.

 \item  {\rm (Large Deviation Principle)} For every $t\in\mathbb{R}$,
 \begin{multline*}
 \lim_{n\to\infty}\frac{1}{n}\log\mu_f\left(\left\{z\in J:\mathrm{sgn}(t)S_ng(z)\ge\mathrm{sgn}(t)n\int_J gd\mu_{f+tg}\right\}\right)\\
  =-t\int_J gd\mu_{f+tg}+P_{top}(f+tg)-P_{top}(f).
 \end{multline*}

  \item The limit $\sigma=0$ if and only if there exists a continuous function $u: (J, \tilde\rho_\infty)\to\mathbb{R}$ such that
  $$g-\int_Jgd\mu_f=u\circ\phi-u.$$

  \item Let $f_1: (J, \tilde\rho_\infty)\to\mathbb{R}$ be another  H\"older continuous function. Then  $\mu_f=\mu_{f_1}$ if and only if there exist $A\in\mathbb{R}$ and a continuous function $u: (J, \tilde\rho_\infty)\to\mathbb{R}$ such that
   $$f-f_1=u\circ\phi-u+A.$$
  \end{enumerate}
\end{theorem}

Let us remark on the assumptions in Theorem \ref{thm:CLT}.  As aforementioned, these assumptions guarantee the existence of a  exponentially contracting metric on $J$. Besides that, the compactness assumption gives desired limits in $J_{\mathbb{C}_p}(\phi)$  for convergence sequences in $J_{\mathbb{C}_p}(\phi)$, see Proposition \ref{prop:coding}. This assumption holds if the Berkovich Julia set of $\phi$ equals to $J_{\mathbb{C}_p}(\phi)$; for such an example, we refer \cite[Theorem A]{Kiwi22} for the tame polynomials in the closure of the tame shift locus.  The density assumption of $J_{\overline{\mathbb{Q}}_p}(\phi)$ in $J_{\mathbb{C}_p}(\phi)$ is very mild in the sense that it can be deduced from the Repelling Density Conjecture \cite[Conjecture 4.3]{Hsia00}. Indeed, the Repelling Density Conjecture claims that the repelling periodic points are dense in the $\mathbb{C}_p$-Julia set, which immediately implies that  $J_{\overline{\mathbb{Q}}_p}(\phi)$ is dense in $J_{\mathbb{C}_p}(\phi)$ since all repelling periodic points are contained in $J_{\overline{\mathbb{Q}}_p}(\phi)$. Although this conjecture is open in general, it has been confirmed in several cases, e.g., certain polynomials (\cite[Th\`eor\'eme]{Bezivin05} and \cite[Corollary C]{Trucco14}), quadratic rational maps (\cite[Proposition 1.2]{Doyle16}), rational maps with positive Lyapunov exponent (\cite[Theorem 2]{Okuyama12}), and rational maps having at least one repelling periodic point (\cite[Th\`eor\'eme 3]{Bezivin01}). In particular, if $\phi$ is sub-hyperbolic but not hyperbolic, then $\phi$ has a repelling periodic point and hence satisfies this density assumption.


Theorem \ref{thm:CLT} can be regarded as a $p$-adic counterpart of the recent results for weakly coarse expanding dynamical system in \cite[Theorem 1.1]{Das21}. We should warn the reader that in the setting of \cite{Das21}, the underlying topological spaces are locally compact, locally connected and path connected; while in our setting, the underlaying space $\mathbb{C}_p$ is not locally compact and totally disconnected.

To end this introduction, we mention that for any rational map $\psi\in\mathbb{C}_p(z)$ of degree at least $2$, there exists a unique invariant probability measure $\nu_\psi$ such that $\psi^\ast\nu_\psi=\deg\psi\cdot\nu_\psi$ and $\mathrm{supp}(\nu_\psi)$ coincide with the Berkovich Julia set of $\phi$, see \cite[Section 10.1]{Baker10}, \cite[Theorem 0.1]{Favre04} and \cite[Th\'eor\`eme A]{Favre10}; then for this measure $\nu_\psi$, a central limit theorem has been established  in \cite[Proposition 3.5]{Favre10}. 

\smallskip
The paper is organized as follows. Section \ref{sec:background} states preliminaries used in the paper. It contains the Fatou and Julia sets defined over distinct fields and basic mapping properties related to derivatives. Section \ref{sec:expansion} is devote to the proof of Theorem \ref{thm:equivalence}. It provides a way to assign each Julia postcritical point  particular numbers corresponding to the aforementioned $\Xi$ and $B$ for a given subhyperbolic rational map.  Section \ref{sec:laws} proves Theorem \ref{thm:CLT}. It establishes a coding of the related $\mathbb{C}_p$-Julia dynamics semiconjugately to symbolic dynamics of finitely many symbols, which may have independent interest.


\section{Background on dynamics}\label{sec:background}

In this section, we provide preliminaries for latter use. We briefly discuss the Fatou and Julia sets in Section \ref{sec:Fatou} and  state some results on injectivity and derivative in Section \ref{sec:inj}.

\subsection{Fatou and Julia sets}\label{sec:Fatou}
Since our argument will take place on several relevant Fatou and Julia sets, we state these sets in this subsection.


For any field extension $L$ of $K$, a rational map $\phi\in K(z)$ of degree at least $2$ induces a dynamical system on the projective space $\mathbb{P}^1_L$; in particular, the map $\phi$ can acts on $\mathbb{P}^1_{\mathbb{C}_p}$.  The $L$-Fatou set $F_L(\phi)$ of $\phi$ is the set of points in $\mathbb{P}^1_L$ having a neighborhood in $\mathbb{P}^1_L$  on which $\{\phi^n\}_{n\ge 1}$ is equicontinuous with respect to the spherical metric induced by the absolute value $|\cdot|$; and the $L$-Julia set $J_L(\phi)$ of $\phi$ is the complement of $F_L(\phi)$ in $\mathbb{P}^1_L$. It follows immediately that if $L'$ is an extension of $L$, then $J_{L}(\phi)\subseteq J_{L'}(\phi)$. 
Since $\mathbb{C}_p$ is totally disconnected, following \cite[Definition 1.3]{Benedetto00}, we say a largest disk contained in $F_{\mathbb{C}_p}(\phi)$ is a \emph{D-component} of $F_{\mathbb{C}_p}(\phi)$.


Recall that a rational map $\phi\in K(z)$ of degree at least $2$ is subhyperbolic if each critical point in $J_{\mathbb{C}_p}(\phi)$ is eventually periodic.
The following result gives the relation among  distinct Fatou sets (w.r.t. Julia sets) for subhyperbolic rational maps. 
\begin{lemma}[{\cite[Theorem 1.3]{Fan21}}]\label{lem:Fatou1}
Let $\phi\in K(z)$ be a subhyperbolic rational map of degree at least $2$. Then 
for any finite extension $L$ of $K$,
$$F_{\mathbb{C}_p}(\phi)\cap\mathbb{P}^1_L=F_L(\phi)\ \ \text{and}\ \ J_{\mathbb{C}_p}(\phi)\cap\mathbb{P}^1_L=J_L(\phi).$$
\end{lemma}

By a result of Benedetto \cite[Theorem 1.2]{Benedetto00} and a result of Rivera-Letelier \cite[Th\'eor\`eme de classification, p. 152]{Rivera03}, if $\phi\in K(z)$ is subhyperbolic, then each D-component of $F(\phi)$ eventually maps into an attracting periodic D-component or an indifferent periodic $D$-component. Also by \cite[Theorem 1.3]{Benedetto00}, there are only finitely many periodic D-components of $F_{\mathbb{C}_p}(\phi)$ intersecting $\mathbb{P}^1_L$ for any given finite extension $L$ of $K$. We summarize these as following.


\begin{lemma}\label{lem:Fatou}
Let $\phi\in K(z)$ be a subhyperbolic rational map of degree at least $2$ and let $L$ be a finite extension of $K$. Then the following hold.
\begin{enumerate}[{\rm (1)}]
\item For any $z\in F_L(\phi)$, there exists $\epsilon:=\epsilon(z)>0$ such that the spherical distance between $\overline{\{\phi^n(z)\}_{n\ge 1}}$ and $J_L(\phi)$ is at least $\epsilon$.
\item There exist finitely many $D$-components of $F_{\mathbb{C}_p}(\phi)$ such that for any $z\in F_L(\phi)$, the forward orbit of $z$ eventually lies in these components.
\end{enumerate}
\end{lemma}

\subsection{Injectivity and derivative}\label{sec:inj}
 Let $\psi\in\mathbb{C}_p(z)$ be a rational map of degree $d\ge 2$. Denote by $\mathrm{Crit}(\psi)$ the set of critical points in $\mathbb{P}^1_{\mathbb{C}_p}$. Then $\mathrm{Crit}(\psi)$ contains $2d-2$ points, counted with multiplicity. We define the postcritical set for $\psi$ as
$$P(\psi):=\bigcup_{n\ge 1} \psi^n(\mathrm{Crit}(\psi)),$$
here we do not take the closure in our propose; and we also set
$$P_0(\psi):=\mathrm{Crit}(\psi)\cup P(\psi).$$
It follows immediately that
\begin{equation}\label{equ:inv}
\psi(P_0(\psi))=P(\psi)\ \ \text{and}\ \ \psi(P(\psi))\subset P(\psi).
\end{equation}

We now state a criterion for the injectivity of $\phi$ on a disk in $\mathbb{C}_p$, which is a direct translation of results in the corresponding Berkovich dynamics (see \cite[Theorems D and F]{Faber13}). Set
\begin{equation}\label{equ:delta}
\delta_{\max}:=\frac{1}{p-1},
\end{equation}
 and write $\mathrm{diam}(D)$ for the diameter of a disk $D\subset\mathbb{C}_p$.
\begin{lemma}\label{lem:inj}
 Let $\psi\in\mathbb{C}_p(z)$ be a nonconstant rational map, and consider a disk $D_1\subset\mathbb{C}_p$. Let $D_2\subset\mathbb{C}_p$  be a disk containing $D_1$ and having diameter at least $p^{\delta_{\max}}\mathrm{diam}(D_1)$. If $D_2\cap\mathrm{Crit}(\psi)=\emptyset$, then $\psi$ is injective on $D_1$.
\end{lemma}

If $\psi$ is injective on a disk, we can compute the derivative $\psi'$  by the following result.
\begin{lemma}[{\cite[Proposition 3.20]{Benedetto19}}]\label{lem:der}
Let $\psi\in\mathbb{C}_p(z)$ be a nonconstant rational map, and consider a disk $D\subset\mathbb{C}_p$ with $\infty\not\in\psi(D)$. If $\psi$ is injective on $D$, then for any $x\in D$,
$$|\psi'(x)|=\frac{\mathrm{diam}(\psi(D))}{\mathrm{diam}(D)}.$$
\end{lemma}

Moreover, we have the following relations on diameters among disks and their images. We say a set $A\subset\mathbb{C}_p$ is an (open) \emph{annulus} if $A=\{r<|z-z_0|<R\}$ for some $z_0\in\mathbb{C}_p$ and $0<r<R<+\infty$; and denote by
$$\mathrm{Mod}(A):=\log_p\frac{R}{r}.$$
On any annulus $A\subset\mathbb{C}_p$, a rational map $\psi\in\mathbb{C}_p$ has inner and outer Weierstrass degrees, see \cite[Definition 3.30]{Benedetto19}. If such degrees coincide, we say simply call this degree is the \emph{Weierstrass degree} of $\psi$ on $A$. 

\begin{lemma}[{\cite[Theorem 3.33]{Benedetto19}}]\label{lem:anu}
Let $\psi\in\mathbb{C}_p(z)$ be a nonconstant rational map. If $\psi$ has Weierstrass degree $m\ge 1$ on an annulus $A\subset\mathbb{C}_p$, then $\psi(A)$ is an annulus and
$$\mathrm{Mod}(\phi(A))=m\mathrm{Mod}(A).$$
In particular, if $\psi$ is injective on $A$, then $\mathrm{Mod}(\phi(A))=\mathrm{Mod}(A)$.
\end{lemma}

We end this section by the following inequality on $\psi'$; its proof is based on Lemmas \ref{lem:inj}, \ref{lem:der} and \ref{lem:anu}. We will repeatedly use it in Section \ref{sec:metric}.
\begin{proposition}\label{prop:der}
Let $\psi\in\mathbb{C}_p(z)$ be a nonconstant rational map, and pick $x\in\mathbb{C}_p\setminus\mathrm{Crit}(\psi)$. Let $D\subset\mathbb{C}_p$ be an (open) disk containing $x$ such that $D\cap\mathrm{Crit}(\psi)=\emptyset$ and $\infty\not\in\psi(D)$. Then
$$|\psi'(x)|\ge \frac{p^{-\delta_{\max}d'}\mathrm{diam}(\psi(D))}{p^{-\delta_{\max}}\mathrm{diam}(D)},$$
where $d'\ge 1$ is the degree of $\psi$ on $D$.
\end{proposition}
\begin{proof}
Let $D_1\subset D$ be an (open) disk containing $x$ such that $\mathrm{diam}(D_1)=p^{-\delta_{\max}}\mathrm{diam}(D)$. Then by Lemma \ref{lem:inj}, we conclude that $\psi$ is injective on $D_1$. Hence by Lemma \ref{lem:der}, we have
$$|\psi'(x)|=\frac{\mathrm{diam}(\psi(D_1))}{\mathrm{diam}(D_1)}=\frac{\mathrm{diam}(\psi(D_1))}{p^{-\delta_{\max}}\mathrm{diam}(D)}.$$
Let $\overline{D}_1$ be the correpsonding close disk for $D_1$. Then the annulus $D\setminus\overline{D}_1$ can be divided into finitely many (distinct) subannuli such that $\psi$ has Weierstrass degree, which is at most $d'$, on each subannulus. Applying  Lemma \ref{lem:anu} to each such annulus, we have
$$\frac{\mathrm{diam}(\psi(D))}{\mathrm{diam}(\psi(D_1))}=\frac{\mathrm{diam}(\psi(D))}{\mathrm{diam}(\psi(\overline{D}_1))}\le\left(\frac{\mathrm{diam}(D)}{\mathrm{diam}(\overline{D}_1)}\right)^{d'}=\left(\frac{\mathrm{diam}(D)}{\mathrm{diam}(D_1)}\right)^{d'}=p^{\delta_{\max}d'},$$
and hence $\mathrm{diam}(\psi(D_1))\ge p^{-\delta_{\max}d'}\mathrm{diam}(\psi(D))$. Thus the conclusion follows.
\end{proof}

\section{Subhyperbolicity and expanding property}\label{sec:expansion}

In this section, we aim to prove Theorem \ref{thm:equivalence}. Fixing the notation as in previous sections, we first establish the following key propositions in Sections \ref{sec:index} and \ref{sec:metric}, and then complete the proof of Theorem \ref{thm:equivalence} in Section \ref{sec:proof1}.
\begin{proposition}\label{prop:main}
Let $\phi\in K(z)$ be a rational map of degree at least $2$ with $\infty\in F_{\mathbb{C}_p}(\phi)$. Set $P_J(\phi):=P(\phi)\cap J_{\mathbb{C}_p}(\phi)$ and denote $s:=\#P_J(\phi)$. Assume  that $s\ge 1$ and each point in $P_J(\phi)$ is eventually fixed. 
Then  there exist  continuous functions $\Xi: \mathbb{C}_p\to(0,1]$ and $B: \mathbb{C}_p\to(0,+\infty)$  and closed disks $\overline{D}_i$, $i=1,\dots, s$, in $\mathbb{C}_p$ with
$\#(\overline{D}_i\cap P_J(\phi))=1$ for each $1\le i\le s$ such that the following hold: denoting by $\overline{U}$ the union of $\overline{D}_i$, $1\le i\le s$,
\begin{enumerate}[{\rm (1)}]
\item the map $\Xi$ is constant on each $\overline{D}_i$ for $1\le i\le m$, and takes value $0$ on $\mathbb{C}_p\setminus\overline{U}$, moreover, for any $z\in \overline{D}_i$ with $\phi(z)\in\overline{U}$,
$$1-\Xi(z)=\deg_{z_0}\phi\cdot(1-\Xi(\phi(z))),$$
 where $z_0$ is the unique point contained in  $\overline{D}_i\cap P_J(\phi)$;
\item the map $B$ is constant in each $\overline{D}_i$, $1\le i\le s$, and takes value $1$ in $\mathbb{C}_p\setminus\overline{U}$; and

\item defining the  function $\theta: \mathbb{C}_p\to(0,+\infty]$ by
$$\theta(z)=\frac{B(z)}{\mathrm{dist}(z,P_J(\phi))^{\Xi(z)}},$$
then $\theta(z)\ge 1$ for any $z\in\mathbb{C}_p$, and  there exist $\Xi_{\mathrm{min}}\in(0,1)$ and $B_{\mathrm{max}}>0$ such that
$1\le \theta(z)\le B_\mathrm{max}/\mathrm{dist}(z,P_J(\phi))^{\Xi_\mathrm{min}}$ for $z\in\overline{U}$.
Moreover, for any finite extension $L$ of $K$, we have
\begin{enumerate}[{\rm (a)}]
\item   $\int_{V\cap L}\theta(z) d\nu_L(z)<+\infty$ for any bounded neighborhood $V\subset\mathbb{C}_p$ of $J_L(\phi)$; and
\item  there exists $\overline{W}\subset\mathbb{C}_p$ of $P_0(\phi)\cap J_{\mathbb{C}_p}(\phi)$ such that
$$\frac{\theta(\phi(z))}{\theta(z)}|\phi'(z)|>1.$$
\end{enumerate}
\end{enumerate}
\end{proposition}

\begin{proposition}\label{prop:main2}
Let $\phi\in K(z)$ be a rational map of degree at least $2$ with $\infty\in F_{\mathbb{C}_p}(\phi)$. If  $P_J(\phi):=P(\phi)\cap J_{\mathbb{C}_p}(\phi)\not=\emptyset$  and each point in $P_J(\phi)$ is eventually fixed, set $\theta$  as in Proposition \ref{prop:main}; and if  $P_J(\phi)=\emptyset$, set $\theta(z)\equiv1$ for any $z\in\mathbb{C}_p$.  Then for any finite extension $L$ of $K$ and for any $\lambda\ge 1$, there exists $N: =N(\phi, L, \lambda)\ge 1$ such that for all $z\in J_L(\phi)$ and all $n\ge N$,
$$\frac{\theta(\phi^{n}(z))}{\theta(z)}|(\phi^{n})'(z)|>\lambda.$$
\end{proposition}

For brevity, we will  use the following notation. For $x\in\mathbb{C}_p$ and $r\ge 0$, denote by $D(x,r)$ (resp.  $\overline{D}(x,r)$) the open (resp. closed) disk (in the $|\cdot|$-metric sense) at $x$ with radius $r$. 

\subsection{Proof of Proposition \ref{prop:main}}\label{sec:index}
In this subsection, we aim to prove Proposition \ref{prop:main}.
Let us first construct desired functions $\Xi$ and $B$. We mainly work on the points near $P_J(\phi)$ and then trivially extend to $\mathbb{C}_p$.

{\bf Step 0: local degrees on $P_J(\phi)$.} For each point $z\in\mathbb{C}_p$, we set
\begin{equation*}\label{equ:alpha}
\alpha(z):=\prod_{z'\in P_0(\phi)\cap\mathrm{orb}_\phi^{-}(z)}\deg_{z'}\phi,
\end{equation*}
where $\mathrm{orb}_\phi^{-}(z):=\{w\in\mathbb{C}_p: \exists \ n\ge 1\ \text{such that}\ \phi^n(w)=z\}$. If $z\in P_J(\phi)$, we have that
\begin{equation*}\label{equ:00}
2\le\alpha(z)<(\deg\phi)^{2\deg\phi-2}<+\infty,
\end{equation*}
since $\phi$ has at most $2\deg\phi-2$ critical points in $\mathbb{C}_p$; in particular,
\begin{equation}\label{equ:000}
\alpha(\phi(c))\ge \deg_c\phi
\end{equation}
for any $c\in\mathrm{Crit}(\phi)$ with $\phi(c)\in P_J(\phi)$.

Let us now focus on a critical orbit. Pick a Julia critical point $z_0\in J_{\mathbb{C}_p}(\phi)\cap\mathrm{Crit}(\phi)$.  Let $\ell\ge1$ be the smallest integer such that $z_\ell:=\phi^\ell(z_0)$ is a fixed point, and for $0\le j\le\ell$, set $z_j:=\phi^j(z_0)$. We will 
consider points near each $z_j$ in the following three steps.


{\bf Step 1: a neighborhood of $z_j$. }
Fix $\epsilon_\ell:=\epsilon(z_\ell)\in(0,1)$ small enough. 
 For $1\le j\le \ell-1$, let $\epsilon_j:=\epsilon(z_j)>0$ such that $\phi(\overline{D}(z_j,\epsilon_j))=\overline{D}(z_{j+1},\epsilon_{j+1})$. Moreover, shrinking $\epsilon_\ell$ if necessary, we can assume that for any $1\le j\le \ell$, the following hold:
\begin{enumerate}
\item If $z_j$ is not a critical point, then $|\phi'(z)|$ is constant in $\overline{D}(z_j,\epsilon_j)$, so
$|\phi(z)-\phi(z_j)|=|a_j||z-z_j|$ for $a_j:=\phi'(z_j)$.

\item  If $z_j$ is a critical point, then $|\phi(z)-\phi(z_j)|=|a_j||z-z_j|^{\deg_{z_j}\phi}$ in $\overline{D}(z_j,\epsilon_j)$ for some $a_j\in K\setminus\{0\}$ independent of $z\in\overline{D}(z_j,\epsilon_j)$.

\item  We have $\epsilon_j\in(0,1)$ sufficiently small such that 
the disks $\overline{D}(z_j,\epsilon_j)$ for $1\le j\le \ell-1$  and $\phi(\overline{D}(z_\ell,\epsilon_\ell))$ are pairwisely disjoint and contain no points in $P_0(\phi)\setminus\{z_1,\dots, z_\ell\}$.
\end{enumerate}


In next two steps, we will work on the union $\cup_{j=1}^{\ell}\overline{D}(z_j,\epsilon_j)$.

{\bf Step 2: points near $z_\ell$. } Since $z_\ell$ is a repelling fixed point of $\phi$, setting
$$\epsilon'_\ell:=\frac{\epsilon_\ell}{|a_\ell|},$$ we have $0<\epsilon'_\ell<\epsilon_\ell$ and $\phi(\overline{D}(z_\ell,\epsilon'_\ell))=\overline{D}(z_\ell,\epsilon_\ell)$. 
 We now pick $\xi_\ell:=\xi(z_\ell)\in(0,1)$ small enough such that
\begin{equation}\label{equ:100}
0<1-\xi_\ell<\frac{1}{\alpha(z_\ell)},
\end{equation}
and choose $\beta_\ell:=\beta(z_\ell)$ such that
$$0<\beta_\ell<\epsilon_\ell.$$
We then observe that for any $z\in\overline{D}(z_\ell,\epsilon_\ell)\setminus\overline{D}(z_\ell,\epsilon'_\ell)$,
\begin{multline}\label{equ:R0}
\frac{1}{\beta_\ell/|z-z_\ell|^{\xi_\ell}}|\phi'(z)|=\frac{1}{\beta_\ell/|z-z_\ell|^{\xi_\ell}}|\phi'(z_\ell)|=\frac{|a_\ell|}{\beta_\ell/|z-z_\ell|^{\xi_\ell}}\\
=\frac{|a_\ell||z-z_\ell|^{\xi_\ell}}{\beta_\ell}>\frac{|a_\ell|{\epsilon'_\ell}^{\xi_\ell}}{\epsilon_\ell}=\frac{|a_\ell|^{1-\xi_\ell}}{\epsilon_\ell^{1-\xi_\ell}}>1,
\end{multline}
and that for any $z\in\overline{D}(z_\ell,\epsilon'_\ell)\setminus\{z_\ell\}$,
\begin{multline}\label{equ:R1}
\frac{\beta_\ell/|\phi(z)-z_\ell|^{\xi_\ell}}{\beta_\ell/|z-z_\ell|^{\xi_\ell}}|\phi'(z)|=\frac{\beta_\ell/|\phi(z)-z_\ell|^{\xi_\ell}}{\beta_\ell/|z-z_\ell|^{\xi_\ell}}|\phi'(z_\ell)|=\frac{\beta_\ell/|\phi(z)-z_\ell|^{\xi_\ell}}{\beta_\ell/|z-z_\ell|^{\xi_\ell}}\frac{|\phi(z)-z_\ell|}{|z-z_\ell|}\\
=\left(\frac{|\phi(z)-z_\ell|}{|z-z_\ell|}\right)^{1-\xi_\ell}
=|a_\ell|^{1-\xi_\ell}>1,
\end{multline}
since both $\epsilon_\ell$ and $\xi_\ell$ are in $(0,1)$ and $|a_\ell|>1$.

{\bf Step 3: points near $z_j$ for $1\le j\le\ell-1$. }

We now pull back $\xi_\ell$ and $\beta_\ell$ at $z_\ell$  in step $1$ inductively along the orbit $z_1,\cdots,z_\ell$ in the following way. Inspired by \eqref{equ:100}, suppose that $\xi_{j+1}:=\xi(z_{j+1})\in(0,1)$ and $\beta_{j+1}:=\beta(z_{j+1})>0$ are well-assigned at $z_{j+1}$ such that
\begin{equation}\label{equ:01}
0<1-\xi_{j+1}<\frac{1}{\alpha(z_{j+1})}.
\end{equation}
We pick $\xi_j:=\xi(z_j)\in(0,1)$ such that
\begin{equation}\label{equ:beta}
1-\xi_j=\deg_{z_j}\phi\cdot (1-\xi_{j+1}),
\end{equation}
and choose $\beta_j=\beta(z_j)$ such that
\begin{equation}\label{equ:B1}
0<\beta_j<\frac{1}{2}\beta_{j+1}|a_j|^{1-\xi_{j+1}}|\deg_{z_j}\phi|.
\end{equation}
Then by \eqref{equ:beta} and  \eqref{equ:B1},  for any $z\in\overline{D}(z_j,\epsilon_j)\setminus\{z_j\}$, we have
\begin{multline*}
\frac{\beta_j|\phi(z)-z_{j+1}|^{\xi_{j+1}}}{|z-z_j|^{\xi_j}\beta_{j+1}}=\frac{\beta_j|a_j|^{\xi_{j+1}}|z-z_j|^{\xi_{j+1}\deg_{z_j}\phi}}{|z-z_j|^{\xi_j}\beta_{j+1}}
< \frac{1}{2}|a_j||\deg_{z_j}\phi||z-z_j|^{\deg_{z_j}\phi-1}
=\frac{|\phi'(z)|}{2}.
\end{multline*}
Hence 
 we conclude that for any $z\in\overline{D}(z_j,\epsilon_j)\setminus\{z_j\}$,
\begin{equation}\label{equ:j}
\frac{\beta_{j+1}/|\phi(z)-z_{j+1}|^{\xi_{j+1}}}{\beta_j/|z-z_j|^{\xi_j}}|\phi'(z)|>2>1.
\end{equation}
Moreover, by \eqref{equ:01} and \eqref{equ:beta}, we have
\begin{equation}\label{equ:degree}
0<1-\xi_j<\frac{\deg_{z_j}\phi}{\alpha(z_{j+1})}\le\frac{1}{\alpha(z_j)}.
\end{equation}

By steps $2$ and $3$, for any point $z_j\in P_J(\phi)$, $1\le j\le \ell$ in the forward orbit of $z_0$, we obtain numbers $\xi(z_j)$ and $\beta(z_j)$. 
Shrinking $\epsilon_\ell$ and $\epsilon_j$ accordingly, we can assume that for any $z\in \overline{D}(z_j, \epsilon_j)$
$$\frac{\beta(z_j)}{|z-z_j|^{\xi(z_j)}}\ge 1,$$
and define
$$\overline{U}_{z_0}:=\cup_{j=1}^{\ell}\overline{D}(z_j,\epsilon_j)$$
Next, we will consider all points near  $P_J(\phi)$ in step 4.

{\bf Step 4: points near $P_J(\phi)$. }
Consider a small neighborhood of the fixed points in $P_J(\phi)$. We repeatedly apply step 1 for all critical orbits in $J_L(\phi)$. Shrinking the initial neighborhood of the fixed points in $P_J(\phi)$ if necessary, for any $c\in\mathrm{Crit}(\phi)\cap J_L(\phi)$ we can obtain $\overline{U}_c$  satisfying (1)-(3) in step 1 such that for any  $z_0, z_0'\in\mathrm{Crit}(\phi)\cap J_L(\phi)$, if $\phi^{m}(z_0)=\phi^{m'}(z_0')$ are contained in $P_L(\phi)$ for some integers $m\ge 0$ and $m'\ge 0$, then the disk component of $\overline{U}_{z_0}$ containing $\phi^{m}(z_0)$ coincides to the the disk component of $\overline{U}_{z'_0}$ containing $\phi^{m'}(z'_0)$.
We set
\begin{equation}\label{equ:U}
\overline{U}=\overline{U}(\phi):=\bigcup_{z_0\in\mathrm{Crit}(\phi)\cap J_L(\phi)}\overline{U}_{z_0}.
\end{equation}
Applying repeatedly steps 2 and 3 for all critical orbits in $\overline{U}$, we can obtain $\xi(z)$ and $\beta(z)$ for any point $z\in P_J(\phi)\cap \overline{U}$.


\smallskip

Observe that for any $z\in \overline{U}$, there is a unique point in $P_J(\phi)$, denoted by $\iota(z)$, such that $\mathrm{dist}(z,P_J(\phi))=|z-\iota(z)|$; moreover, by the choice of $ \overline{U}$, it follows that $\phi(\iota(z))=\iota(\phi(z))$.
Now we can define $\Xi=\Xi_\phi:\mathbb{C}_p\to[0,1)$ by
\begin{equation*}\label{equ:xi}
\Xi(z):=\begin{cases}
\xi(\iota(z))\ &\text{if}\ z\in \overline{U},\\
0\ &\text{if otherwise},
\end{cases}
\end{equation*}
and define $B=B_\phi:\mathbb{C}_p\to(0,1]$ by
\begin{equation*}\label{equ:B}
B(z):=\begin{cases}
\beta(\iota(z))\ &\text{if}\ z\in \overline{U},\\
1\ &\text{if otherwise}.
\end{cases}
\end{equation*}




From the above construction, we obtain Proposition \ref{prop:main} (1) and (2)  immediately.  For Proposition \ref{prop:main} (3), by the finiteness of $P(\phi)$, we set $\Xi_{\mathrm{min}}=\min\{\Xi(z):z\in P(\phi)\}$ and $B_{\mathrm{max}}=\max\{B(z):z\in P(\phi)\}$. Then we obtain the desired bound for $\theta$. 




 We now prove Proposition \ref{prop:main} (3a) in the following result. Denote by $\pi_L$ a uniformizer of $L$, and let $S_r(x):=\{y\in\mathbb{C}_p: |y-x|=r\}\subset\mathbb{C}_p$ be the sphere at $x$ with radius $r$.
\begin{lemma}\label{lem:integral}
If $V\subset\mathbb{C}_p$ is a bounded neighborhood  of $J_L(\phi)$, then
$$\int_{V\cap L} \theta(z)d\mu_L(z)<+\infty.$$
\end{lemma}
\begin{proof}
It suffices to show that for any $z_0\in P_J(\phi)\cap L$, there is a small open neighborhood $D\subset\mathbb{C}_p$ of $z_0$ such that $\int_{D\cap L}\theta(z) d\mu_L(z)<\infty$.
Recall $\overline{U}$ as in \eqref{equ:U} and consider a small open disk $D\subset\overline{U}$ of $z_0$. We check that $D$ is a desired neighborhood. For $z\in D\setminus\{z_0\}$, denote by $n:=n(z)$ the integer such that $|z-z_0|=|\pi_L|^{n}$. Then in the sphere $S_{|z-z_0|}(z_0)$, we have
\begin{equation}\label{equ:z}
\theta(z)=\frac{B(z)}{\mathrm{dist}(z,P_J(\phi))^{\Xi(z)}} \le\frac{B_{\max}}{|z-z_0|^{\Xi(z)}}=B_{\max}\left(\frac{1}{|\pi_L|}\right)^{n\Xi(z_0)}.
\end{equation}

Let $A_L$ be cardinality of the residue classes of $L$, we have
$$\nu_L(S_{|z-z_0|}(z_0))=\frac{A_L-1}{A_L}\nu_L(\overline{D}(z_0,|z-z_0|))=\frac{A_L-1}{A_L} A_L^{-n}.$$
Noting that $A_L\ge p$ and $|\pi_L|\ge p^{-1}$
we conclude that $|\pi_L|^{\Xi(z_0)}A_L>1$ since $\Xi(z_0)\in(0,1)$.
Using \eqref{equ:z}, we compute
\begin{multline*}
\int_{D\cap L}\theta(z) d\nu_L(z)=\sum_{\substack{r\in |L^\times|\\ r<\mathrm{diam}(W)}}\int_{S_r(z_0)\cap L}\theta(z) d\mu_L(z)\\\le B_{\max}\frac{A_L-1}{A_L} \sum_{n\ge 1}\left(\left(\frac{1}{|\pi_L|}\right)^{n\Xi(z_0)}A_L^{-n}\right)
=B_{\max}\frac{A_L-1}{A_L} \sum_{n\ge 1}\left(\frac{1}{|\pi_L|^{\Xi(z_0)}A_L}\right)^n<+\infty.
\end{multline*}
The conclusion follows.
\end{proof}

To the end of this section, for brevity,  we use the following definition

\begin{definition}[{The $\theta$-Derivative}]\label{def:Dtheta}
Let $\theta$ be as above for the subhyperbolic rational map $\phi\in K(z)$, and let $\psi\in\mathbb{C}_p(z)$ be a non-constant rational map.
We say $\psi$ is \emph{$\theta$-differentiable} at  $z\in\mathbb{C}_p$ if the following limit exists
$$\lim_{n\to\infty}\frac{\theta(\psi(z_n))|\psi'(z_n)|}{\theta(z_n)},$$
where $\{z_n\}_{n\ge1}\subset\mathbb{C}_p$ is any sequence converging to $z$ as $n\to\infty$;
and call the limit is the \emph{$\theta$-derivative} of $\psi$ at $z$.
\end{definition}
Remark that the $\theta$-derivative of $\phi$ at some point could be $+\infty$.

Observe that for any integer $m\ge 1$, the $\theta$-derivative of $\phi^m$ at any $z\in\mathbb{C}_p$ exists, and moreover, if  $z\notin P_0(\phi^m)$,
$$D_\theta\phi^m(z)=\frac{\theta(\phi^m(z))|(\phi^m)'(z)|}{\theta(z)}.$$


To establish Proposition \ref{prop:main} (3b),  we begin to show an expanding property of $\phi$ near the Julia critical points not contained in $P(\phi)$ with respect to $\theta$.

\begin{lemma}\label{lem:crit}
Suppose that  $z_0\in\mathrm{Crit}(\phi)\setminus P_J(\phi)$ is a point contained in $J_{\mathbb{C}_p}(\phi)$. Then there exists an $\epsilon:=\epsilon_{z_0}>0$ such that  for any $z\in\overline{D}(z_0,\epsilon)$,
$$D_\theta\phi(z)>1.$$
\end{lemma}
\begin{proof}
Since $z_0\not\in P_J(\phi)$, we can choose $\epsilon>0$ small enough so that $\theta(z)=\theta(z_0)=1$ for all $z\in\overline{D}(z_0,\epsilon)$ and $\phi(\overline{D}(z_0,\epsilon))\subset\overline{U}$.
Writing $d_0:=\deg_{z_0}\phi\ge2$ and shrinking $\epsilon$ if necessary, we have that $|\phi(z)-\phi(z_0)|=|a||z-z_0|^{d_0}$ for some $a:=a_{z_0}\in K\setminus\{0\}$ and that $\phi(\overline{D}(z_0,\epsilon))\cap P_J(\phi)=\{\phi(z_0)\}$. We conclude that
$$\theta(\phi(z))=\frac{B_{z_0}}{(|a||z-z_0|^{d_0})^{\Xi(\phi(z_0))}},$$
where $B_{z_0}:=B(\phi(z))$ is constant in $\overline{D}(z_0,\epsilon)$.
Moreover, noting that 
$$|\phi'(z)|=|ad_0||z-z_0|^{d_0-1},$$
we conclude that
$$D_\theta\phi(z)=\theta(\phi(z))|\phi'(z)|=\frac{B_{z_0}|ad_0||z-z_0|^{d_0-1}}{(|a||z-z_0|^{d_0})^{\Xi(\phi(z_0))}}.$$
Since $\Xi(\phi(z_0))>(d_0-1)/d_0$ by \eqref{equ:000} and \eqref{equ:degree}, further shrinking $\epsilon$ if necessary, we conclude that $\theta(\phi(z))|\phi'(z)|>1$. Thus the conclusion follows.
\end{proof}

 Recall $\overline{U}$ in \eqref{equ:U}. We set
 \begin{equation}\label{equ:W}
 \overline{W}:=\overline{U}\bigcup\left(\bigcup_{z_0\in(\mathrm{Crit}(\phi)\setminus P_J(\phi))\cap J_{\mathbb{C}_p}(\phi)}\overline{D}(z_0,\epsilon_{z_0})\right),
 \end{equation}
 where $\epsilon_{z_0}$ is as in Lemma \ref{lem:crit}. Observe that from the construction of $\Xi$ and $B$ (see \eqref{equ:R0}, \eqref{equ:R1} and \eqref{equ:j}), for any $z\in \overline{U}$,
\begin{equation}\label{equ:p}
D_\theta\phi(z)>1
\end{equation}
We then obtain Proposition \ref{prop:main} (3b) by \eqref{equ:p} and Lemma \ref{lem:crit}.

 To end this subsection, we state the following consequence of Proposition \ref{prop:main} (3b).

\begin{corollary}\label{cor:lower}
There exists $C_0>0$ such that for any $z\in J_L(\phi)$,
$$D_\theta\phi(z)>C_0.$$
Moreover, there exists $\widehat C_0>1$ such that for any $z\in \overline{W}\cap L$,
$$D_\theta\phi(z)>\widehat C_0.$$
\end{corollary}
\begin{proof}
For  $z\in J_L(\phi)\setminus\overline{W}$, we have  $D_\theta\phi(z)=\theta(\phi(z))|\phi'(z)|>0$. Then by the continuity of $\theta(\phi(z))|\phi'(z)|$ and the compactness of $J_L(\phi)\setminus\overline{W}$, the conclusion follows.  For  $z\in\overline{W}\cap L$, conclusion follows from Proposition \ref{prop:main} (3b).
\end{proof}

\subsection{Expansion under iteration}\label{sec:metric}
In this subsection, we fix the notation as in previous subsections and aim to prove Proposition \ref{prop:main2}.  If $P_J(\phi)\not=\emptyset$, we will first upgrade the constant $C_0>0$ in Corollary \ref{cor:lower} to some constant $C_1>1$ by considering an iteration of $\phi$, where the number of iteration depends on the points $z\in J_L(\phi)$.  Noting that $\overline{W}$ defined  as in \eqref{equ:W} is a disjoint union finitely many disks, we let $\epsilon_{\mathrm{in}}>0$ be the inradius of $\overline{W}$, that is the maximal number such that there is a disk of radius $\epsilon_{\mathrm{in}}$  contained in $\overline{W}$. For brevity, we extend the Definition \eqref{def:Dtheta} to the case that $P_J(\phi)=\emptyset$ by taking $\theta\equiv1$.

\begin{lemma}\label{lem:exp1}
Let $\phi\in K(z)$ be a rational map of degree at least $2$ with $\infty\in F_{\mathbb{C}_p}(\phi)$. Then there exists $C_1>1$ satisfying the following property:  for any $z\in J_L(\phi)$, there exists an integer $n_z\ge 1$, such that
$$D_\theta\phi^{n_z}(z)>C_1.$$
Moreover, we can choose above $n_z$ to be locally constant  in $z$.
\end{lemma}

\begin{proof}
Let us assume that $J_L(\phi)\not=\emptyset$; for otherwise the conclusion holds trivially. By Corollary \ref{cor:lower}, for $z\in\overline{W}\cap L$, we can choose $n_z=1$. Now we work on the points $z\in J_L(\phi)\setminus\overline{W}$. It follows that $\theta(z)=1$. Pick an $0<\epsilon_{\min}<\epsilon_{\mathrm{in}}$ such that the $\epsilon_{\min}$-neihborhood of $J_L(\phi)$ in $\mathbb{C}_p$ contains no poles  and no critical points of $\phi$ in $\mathbb{C}_p\setminus J_L(\phi)$.  Denoting by $d:=\deg\phi\ge 2$ and recalling $\delta_{\max}$ as in \eqref{equ:delta}, we consider a disk neighborhood $W_z\subset\mathbb{C}_p\setminus\overline{W}$ of $z$ such that
\begin{equation}\label{equ:assume}
0<\mathrm{diam}(W_z)<p^{-\delta_{\max}d^{2d-2}}\epsilon_{\min},
\end{equation}
and further shrinking $W_z$ if necessary, we can assume that
$\deg_{W_z}\phi=1,$
since $z\not\in\mathrm{Crit}(\phi)$.
Noting that $W_z\cap \overline{W}=\emptyset$, we conclude that  $\theta(w)=1$ for any $w\in W_z$. Hence  for any $w\in W_z$ and for any $n\ge 1$,
  \begin{equation}\label{equ:w}
  D_\theta\phi^n(w)\ge |(\phi^n)'(w)|,
  \end{equation}
  since $\theta(\phi^n(w))\ge 1$ by Proposition \ref{prop:main} (3).

 Since $z\in J_L(\phi)\setminus \overline{W}\subset J_{\mathbb{C}_p}(\phi)\setminus\overline{W}$, there exists a smallest integer $\ell_z\ge 1$ such that $\phi^{\ell_z}(W_z)\cap\mathrm{Crit}(\phi)\not=\emptyset$ or $\deg_{\phi^{\ell_z}(W_z)}\phi\ge 2$ (see \cite[Corollary 5.21]{Benedetto19}); for otherwise, $\mathrm{Crit}(\phi)$ is contained in the exceptional set of $\phi$ and hence $\phi$ has good reduction (see \cite[Theorem 1.19]{Benedetto19}), which implies that $J_{\mathbb{C}_p}(\phi)$ is empty (see \cite[Theorem 5.11]{Benedetto19}). Observe that $\ell_z$ is the largest integer such that $\phi^{\ell_z}$ is injective on $W_z$ that only depends on the choice of $W_z$. Hence $|(\phi^{\ell_z})'(w)|$ is non-zero and constant for $w\in W_z$, denoted by $\Delta_z$.

 If there exists (smallest) $1\le j_0\le \ell_z$ such that $\phi^{j_0}(W_z)$ contains a pole or a critical point in $\mathbb{C}_p\setminus J_L(\phi)$, then $\mathrm{diam}(\phi^{j_0}(W_z))\ge\epsilon_{\min}$ since $\phi^{j_0}(W_z)$ also intersects $J_{\mathbb{C}_p}(\phi)$.   Observing that $\phi^{j_0}$ is injective on $W_z$, by Lemma \ref{lem:der}, \eqref{equ:assume} and \eqref{equ:w},
 we conclude that for any $w\in W_z$,
 $$ D_\theta\phi^{j_0}(w)\ge |(\phi^{j_0})'(w)|=\frac{\mathrm{diam}(\phi^{j_0}(W_z))}{\mathrm{diam}(W_z)}\ge \frac{\epsilon_{\min}}{p^{-\delta_{\max}d^{2d-2}}\epsilon_{\min}}>1.$$

  Now we work on the case that $\phi^{j}(W_z)$ contains neither a pole nor a critical point in $\mathbb{C}_p\setminus J_L(\phi)$ for any $1\le j\le\ell_z$.  We will show that there exist a subset $V_z\subset W_z$ containing $z$ and an integer $n_z\ge 1$ only depending on $V_z$ such that  $D_\theta\phi^{n_z}(w)>1$ for any $w\in V_z\cap L$.

  Recall $\widehat C_0>1$ for points in $\overline{W}\cap L$ as in Corollary \ref{cor:lower}. Pick $m_0\ge1$ such that
$$\widehat C_0^{m_0}\Delta_z>1.$$
Then for each point $c\in\mathrm{Crit}(\phi)\cap J_L(\phi)$, we denote by $D_c\subset\overline{W}$ the maximal disk containing $c$ such that $\phi^{m_0}(D_{c})\subset\overline{W}$.

 If $\phi^{\ell_z}(z)\in D_c$ for some  $c\in\mathrm{Crit}(\phi)\cap J_L(\phi)$, consider the maximal disk neighborhood $V_z\subset W_z$ of $z$ such that $\phi^{\ell_z}(V_z)\subset D_c$. Setting $n_z=\ell_z+m_0$, we compute that for any $w\in V_z\cap L$,
  \begin{align*}
D_\theta\phi^{n_z}(w)&=|(\phi^{\ell_z})'(w)||(\phi^{n_z-\ell_z})'(\phi^{\ell_z}(w))|\frac{\theta(\phi^{n_w}(w))}{\theta(\phi^{\ell_z}(w))}\\
&=|(\phi^{\ell_z})'(w)|\left(\prod_{j=1}^{m_0}|\phi'(\phi^{\ell_z+j-1}(w))|\right)\prod_{j=1}^{m_0}\frac{\theta(\phi^{\ell_z+j}(w))}{\theta(\phi^{\ell_z+j-1}(w))}\\
&=\Delta_z\prod_{j=1}^{m_0}D_\theta(\phi)(\phi^{\ell_z+j-1}(w))
\ge\Delta_z\widehat C_0^{m_0}>1.
\end{align*}
The above first inequality follows from Corollary \ref{cor:lower} since $\phi^{\ell_z+j-1}(w)\in \overline{W}\cap L$.

  If $\phi^{\ell_z}(z)\not\in D_c$ for any $c\in\mathrm{Crit}(\phi)\cap J_L(\phi)$, consider the maximal (open) disk $D\subset\mathbb{C}_p$ such that $\phi^{\ell_z}(z)\in D$ and $D\cap\mathrm{Crit}(\phi)=\emptyset$. Let $V_z\subset W_z$  the maximal disk neighborhood of $z$ such that $\phi^{\ell_z}(V_z)\subset D$.
  By Lemma \ref{lem:inj} and the choice of $\ell_z$, we have
  \begin{equation}\label{equ:D}
  \mathrm{diam}(\phi^{\ell_z}(V_z))>p^{-\delta_{\max}}\mathrm{diam}(D).
  \end{equation}
  Denote by $k_0\ge 0$ the smallest integer such that $\phi^{k_0}(D)\cap\overline{W}=\emptyset$. Then
    \begin{equation}\label{equ:d}
    \mathrm{diam}(\phi^{k_0}(D))\ge\epsilon_{\min}.
    \end{equation}
    Applying Proposition \ref{prop:der} to $D$ and $\phi^{k_0}$, since $\phi^{k_0}$ has degree at most $d^{2d-2}$ on $D$, we obtain that for any $w\in V_z$,
 \begin{equation}\label{equ:M}
 |(\phi^{k_0})'(\phi^{\ell_z}(w))|\ge \frac{p^{-\delta_{\max}d^{2d-2}}\mathrm{diam}(\phi^{k_0}(D))}{p^{-\delta_{\max}}\mathrm{diam}(D)}\ge \frac{p^{-\delta_{\max}d^{2d-2}}\epsilon_{\min}}{\mathrm{diam}(\phi^{\ell_z}(V_z))},
 \end{equation}
  where the last inequality follows from \eqref{equ:D} and \eqref{equ:d}. Setting $n_z=\ell_z+k_0$,
  by  Lemma \ref{lem:der}, \eqref{equ:assume} and \eqref{equ:w} and \eqref{equ:M}, we compute that for any $w\in V_z$,
  \begin{align*}
  D_\theta\phi^{n_z}(w)&\ge| (\phi^{n_z})'(w)|=| (\phi^{\ell_z})'(w)|\cdot|(\phi^{k_0})'(\phi^{\ell_z}(w))|
  \ge\frac{\mathrm{diam}(\phi^{\ell_z}(V_z))}{\mathrm{diam}(V_z)}\cdot\frac{p^{-\delta_{\max}d^{2d-2}}\epsilon_{\min}}{\mathrm{diam}(\phi^{\ell_z}(V_z))}\\
& \ge\frac{p^{-\delta_{\max}d^{2d-2}}\epsilon_{\min}}{\mathrm{diam}(W_z)} >\frac{p^{-\delta_{\max}d^{2d-2}}\epsilon_{\min}}{p^{-\delta_{\max}d^{2d-2}}\epsilon_{\min}}=1.
  \end{align*}

 Then the conclusion follows from the compactness of $J_L(\phi)$.
\end{proof}

We now promote Lemma \ref{lem:exp1} to obtain a uniform number of iteration, which establishes Proposition \ref{prop:main2}.

\begin{proof}[Proof of Proposition \ref{prop:main2}]
If $P_J(\phi)\not=\emptyset$, consider $C_0>0$ as in Corollary \ref{cor:lower}; and if $P_J(\phi)=\emptyset$, consider $C_0:=\min\{|\phi'(z)|: z\in J_L(\phi)\}>0$. Without loss of generality, we can assume that $0<C_0<1$. Let $C_1>1$ and $n_z\ge 1$  as in Lemma \ref{lem:exp1}. Since $n_z$ is locally constant, by the compactness of $J_L(\phi)$, there exists $n_{\max}=\max\{n_z: z\in J_L(\phi)\}$. Choose $N>0$ a multiple of $n_{\max}$ large enough such that $C_1^{N/n_{\max}}C_0^{n_{\max}}>\lambda$. Now pick $z\in J_L(\phi)$ and apply Lemma \ref{lem:exp1} to the forward orbit of $z$. For any $n\ge N$, we conclude that there exists $j\ge 1$ such that $S_{j-1}:=\sum_{i=0}^{j-1}n_{\phi^i(z)}\le n$ and $S_j:=\sum_{i=0}^{j}n_{\phi^i(z)}> n$.  Note that $0<n-S_{j-1}<S_j-S_{j-1}\le n_{\max}$ and $N/n_{\max}\le j$. We compute
$$D_\theta\phi^{n}(z)>C_1^j\prod_{\ell=S_{j-1}}^{n-1}D_\theta\phi(\phi^\ell(z))\ge C_1^jC_0^{n_{\max}}\ge C_1^{N/n_{\max}}C_0^{n_{\max}}> \lambda.$$
Then the conclusion follows.
\end{proof}

\subsection{Proof of Theorem \ref{thm:equivalence}}\label{sec:proof1}

Fix the notation as in previous subsections. Since $\phi$ is subhyperbolic, there exists a smallest integer $m\ge 1$ such that each critical point in $J_{\mathbb{C}_p}(\phi)$ is eventually fixed by $\phi^m$. Construct the function $\theta$ for $\phi^m$  as in Proposition \ref{prop:main}. Now pick $\lambda\ge 1$, and let $N_\lambda$ as in  Proposition \ref{prop:main2} for $\phi^m$.  Set $N=mN_\lambda$. We can assume that $N\ge 2$ by considering the iteration of $\phi^{N}$ if necessary.  For $z\in J_L(\phi)$, define
\begin{equation}\label{rho}
\rho(z)=\rho_{L,\lambda}(z):=\left(\prod_{j=0}^{N-1}\theta(\phi^j(z))\right)^{1/N}\cdot\prod_{j=0}^{N-2}|\phi'(\phi^j(z))|^{(N-1-j)/N}.
\end{equation}
Let $V:=V_\rho\subset\mathbb{C}_p$ be a neighborhood of $J_L(\phi)$ not containing the $n$-th iterated preimages of $\infty$ for $0\le n\le N$. Shrinking $V$ if necessary, we can extend $\rho$ to a well-defined continuous map on $V$.
\begin{proposition}\label{prop:exp}
Consider $\rho$ and $V$ as above. Then $\rho$ is admissible. Moreover,
\begin{enumerate}[{\rm (1)}]
\item For $z,w\in V\cap L$, let $\overline{D}_L:=\overline{D}(z,|z-w|)\cap L$ be the minimal closed disk in $L$ containing $z$ and $w$. Then $\tilde\rho_L(z,w):=\int_{\overline D}\rho d\mu_L$ defines a distance.
\item for all $z\in J_L(\phi)$,
$$\frac{\rho(\phi(z))}{\rho(z)}|\phi'(z)|>\lambda^{1/N}\ge1.$$
\end{enumerate}
\end{proposition}
\begin{proof}
Let us first show $\rho(z)$ is admissible. By Proposition  \ref{prop:main} (3), we have that $\rho$ satisfies Definition \ref{def:admissible} (2) and (3) since $\prod_{j=0}^{N-2}|\phi'(\phi^j(z))|^{N-1-j}$ is bounded above by the choice of $V$. Now let us see that $\rho$ satisfies Definition \ref{def:admissible} (1). It suffices to consider the case that $z=z_0\in\mathrm{Crit}(\phi)\cap J_L(\phi)$. For $j\ge 0$, set $z_j:=\phi^j(z)$.

Let  $\ell\ge 0$ be the largest integer such that there exists $0=j_0<j_1<\dots<j_\ell\le N-1$ satisfying that
$z_{j_i}\in \mathrm{Crit}(\phi)\cap J_L(\phi)$, and denote by $E$ the set of integers between $0$ and $N-1$ not equal to $j_i$ for any $0\le i\le \ell$.
Then $d_i:=\deg_{z_{j_i}}\phi\ge 2$ for $0\le i\le \ell$ and $\deg_{z_j}\phi=1$ for $j\in E$.  It follows from Proposition \ref{prop:main} (1) that for $0\le i\le \ell$,
\begin{equation}\label{equ:jx}
1-\Xi(z_{j_i})=d_i(1-\Xi(z_{j_i+1}))
\end{equation}
and  for any $j\in E$,
\begin{equation}\label{equ:jx0}
 \Xi(z_j)=\Xi(z_{j+1}).
 \end{equation}
 Thus for $0\le i\le \ell-1$,
 \begin{equation}\label{equ:jx1}
1-\Xi(z_{j_i})=d_i(1-\Xi(z_{j_{i+1}})).
\end{equation}

Then for $z'$ sufficiently close to $z$, applying \eqref{equ:jx0}, we have that for any $0\le i\le \ell$ and any $j_i<j\le j_{i+1}$,
$$\theta(\phi^j(z'))=\frac{B(z_j)}{|\phi^j(z')-\phi^j(z_0)|^{\Xi(z_{j_{i+1}})}}=\frac{B_{z_j}}{|z'-z_0|^{d_0d_1\dots d_i\Xi(z_{j_{i+1}})}}$$
 for some constant $B_{z_j}>0$,
and hence
$$\prod_{j=j_i+1}^{j_{i+1}}\theta(\phi^j(z'))=\frac{C_{i+1}(z_0)}{|z'-z_0|^{\Xi(z_{j_{i+1}})(j_{i+1}-j_i)d_0d_1\cdots d_i}}$$
for some constant $C_{i+1}(z_0)>0$. It follows that
$$\prod_{j=0}^{N-1}\theta(\phi^j(z'))=\frac{C_{z_0}}{|z'-z_0|^{\Xi(z_0)}(\prod\limits_{i=0}^{\ell-1}|z'-z_0|^{d_0d_1\dots d_i\Xi(z_{j_{i+1}})(j_{i+1}-j_i)})|z'-z_0|^{d_0d_1\dots d_{\ell}\Xi(z_{j_\ell+1})(N-1-j_\ell)}}.$$
for some constant $C_{z_0}>0$. Applying \eqref{equ:jx} for $i=\ell$ and then inductively applying \eqref{equ:jx1}, we conclude that
$$\prod_{j=0}^{N-1}\theta(\phi^j(z'))=\frac{C_{z_0}}{|z'-z|^{\Xi(z_0)N+\sum_{i=0}^{\ell}d_0\cdots d_{i-1}(d_i-1)(N-1-j_i)}}.$$

Observe that
\begin{align*}
\prod_{j=0}^{N-2}|\phi'(\phi^j(z'))|^{(N-1-j)/N}&=C'_{z_0}\prod_{i=0}^{\ell}|\phi'(\phi^{j_i}(z'))|^{(N-1-j_i)/N}\\
&=\widehat C'_{z_0}\left(|z'-z_0|^{\sum_{i=0}^{\ell}d_0\cdots d_{i-1}(d_i-1)(N-1-j_i)}\right)^{1/N},
\end{align*}
for some constants $C'_{z_0}>0$ and $\widehat C'_{z_0}>0$.
Thus we conclude that $\rho(z)>0$ for $z\in\mathrm{Crit}(\phi)\cap J_L(\phi)$. Hence Definition \ref{def:admissible} (1) by the compactness of $J_L(\phi)$ and the continuity of $\rho$ in $V$. Hence $\rho$ is admissible.

To see that  $\tilde\rho(z,w)$ is a metric, by Lemma \ref{lem:integral}, we only need to show $\tilde\rho(z,w)=0$ if and only if $z=w$. It follows immediately from that $\rho$ has a positive lower bound and $\mu_L(\overline{D}_L)=0$ if and only if $z=w$.

To see the expansion with respect to $\rho$, we note that
$$\frac{\rho(\phi(z))}{\rho(z)}|\phi'(z)|=\left(D_\theta\phi^N(z)\right)^{1/N}.$$
Hence by  Proposition \ref{prop:main2}, we obtain the desired conclusion.
\end{proof}

Now we can prove Theorem \ref{thm:equivalence}.
\begin{proof}[Proof of Theorem \ref{thm:equivalence}]
Suppose that statement (1) holds, then by  Proposition \ref{prop:exp}, statement (2) follows immediately. Now assume that statement (2) holds.  From \eqref{equ:expansion}, for any finite extension $L\not=\mathbb{Q}_p$ of $K$ containing all the critical points of $\phi$, by Lemma \ref{lem:Fatou1} (2), we conclude that $J_{\mathbb{C}_p}(\phi)\cap\mathrm{Crit}(\phi)=J_L(\phi)\cap\mathrm{Crit}(\phi)$. Thus each critical value in $J_{\mathbb{C}_p}(\phi)$ is an exceptional point of the associated admissible function in a neighborhood of  $J_L(\phi)$. It follows that each critical point in $J_{\mathbb{C}_p}(\phi)$ is eventually periodic.
\end{proof}

To end this section, we state the  following  properties for $\tilde\rho_L(z,w)$ in Proposition \ref{prop:exp}. First, the distance $\tilde\rho_L(z,w)$ satisfies non-archimedean property.
\begin{corollary}\label{cor:non}
Pick $z_1,z_2, z_3\in L$. Then $\tilde\rho_L(z_2,z_3)\le\max\{\tilde\rho_L(z_1,z_2),\tilde\rho_L(z_2,z_3)\}$.
\end{corollary}
\begin{proof}
It follows immediately from the fact that $\overline{D}(z_2,|z_2-z_3|)$ is contained in either $\overline{D}(z_2,|z_1-z_2|)$ or $\overline{D}(z_3,|z_1-z_3|)$.
\end{proof}

Moreover, the  distance $\tilde\rho_L(z,w)$ is finer than the non-archimedean distance induced by the absolute value in $L$ in the following sense.
\begin{corollary}\label{cor:fine}
Pick two sequences $\{x_n\}_{n\ge 1}$ and $\{y_n\}_{n\ge 1}$ in $V\cap L$. If $\tilde\rho_L(x_n,y_n)\to 0$ as $n\to\infty$, then $|x_n-y_n|\to 0$.
\end{corollary}
\begin{proof}
Suppose on the contrary that $|x_n-y_n|\not\to 0$. Passing to a subsequence if necessary, we have that the disks $\overline{D}(x_n,|x_n-y_n|)$ converge to a nontrivial disk $\overline{D}\subset V\cap L$. Then $\int_{\overline D}\rho d\mu_L\not=0$. Thus $\tilde\rho_L(x_n,y_n)\not\to 0$. This is a contradiction.
\end{proof}

\section{Statistical laws for subhyperbolic maps}\label{sec:laws}
In this section, let $\phi\in K(z)$ be a subhyperbolic rational map of degree at least $2$. Up to conjugacy, we can assume that  $\infty\in F_{\mathbb{C}_p}(\phi)$.  
We will also take the following assumption.

\begin{assumption}\label{assume1}
$J:=J_{\mathbb{C}_p}(\phi)$ is compact, and $J_{\overline{\mathbb{Q}}_p}(\phi)$ is dense in $J$.
\end{assumption}

\noindent Throughout this section, we fix a finite cover $\mathcal{U}_0$ of $J$ with disks in $\mathbb{C}_p$, and for $n\ge 1$, denote by $\mathcal{U}_n$ the cover of $J$ consisting of the maximal disks in $\phi^{-n}(U)$ for all $U\in\mathcal{U}_0$. To abuse notation, for an extension $L$ of $K$, we write $\mathcal{U}_n\cap L:=\{U\cap L: U\in\mathcal{U}_n\}$.

\subsection{Exponentially contracting metrics}\label{sec:exp}
 In this subsection, we use the following definition for exponentially contracting metrics on $J$.
\begin{definition}\label{def:con}
A metric $\tilde\rho$ on $J$ is \emph{exponentially contracting} if there exists $C>0$ and $0<\Lambda<1$ such that
$$\sup\{\mathrm{diam}_{\tilde\rho}(U): U\in\mathcal{U}_n\}\le C\Lambda^n$$
\end{definition}

We now show that there is an exponentially contracting metric on $J$. This is an application of the existence of expanding admissible metrics in finite extensions of $K$ (see Section \ref{sec:proof1}) and the following metrization lemma.
\begin{lemma}[{Frink's metrization lemma \cite[Lemma 4.6.2]{Przytycki10}}]\label{lem:Frink}
Let $X$ be a topological space, and let $\{\Omega_n\}_{n\ge 0}$ be a sequence of open neighborhoods in $X\times X$ of the diagonal $\Delta\subset X\times X$ such that
\begin{enumerate}[{\rm (a)}]
\item $\Omega_0=X\times X$,
\item $\cap_{n=0}^\infty\Omega_n=\Delta$, and
\item $\Omega_n\circ\Omega_n\circ\Omega_n\subset\Omega_{n-1}$ for any $n\ge 1$, where for two subsets $R, S\subset X\times X$,
$$R\circ S=\{(x,y)\in X\times X:\exists z\in X \ \ \text{such that}\ \ (x,z)\in R, (z,y)\in S\}.$$
\end{enumerate}
Then there exists a metric $\rho_\infty$ on $X$, compatible with the topology, such that
$$\Omega_n\subset\{(x,y)\in X\times X: \rho_\infty(x,y)<2^{-n}\}\subset\Omega_{n-1}$$
for any $n\ge 1$.
\end{lemma}

For $\epsilon>0$ and a subset $E$ of $\mathbb{C}_p$, we say an \emph{$\epsilon$-neighborhood} of $E$ is the set of points $x\in\mathbb{C}_p$ such that $\mathrm{dist}(x,E)<\epsilon$. By  assumption \ref{assume1}, there exists a sequence $\{L_m\}_{m\ge 1}$ of finite extensions of $K$ such that $J$ is contained in a $1/2^m$-neighborhood of $J_{L_m}(\phi)$.
\begin{proposition}\label{prop:con}
Let $\phi$ and $\{L_m\}_{m\ge 1}$ be as above. Then there exists an exponentially contracting metric $\tilde\rho_\infty$ on $J$.
\end{proposition}
\begin{proof}
The proof is a variant of \cite[Proposition 2.14]{Das21}. Denote by $V_{m,n}$ the set of $(x,y)\in J\times J$ such that $x,y$ are contained in a $1/2^m$- neighborhood of $U$ for some $U\in \mathcal{U}_n\cap L_m$. Recall that $\tilde\rho_m$ is the expanding distance function on $J_{L_m}(\phi)$ as in Proposition \ref{prop:exp}, and let $\eta_m$ be the Lebesgue number of $\mathcal{U}_0\cap L_m$ with respect to $\tilde\rho_m$. Then there exists an integer $M_m\ge 1$ such that  for all $\ell\ge M_m$,
$$\sup\{\mathrm{diam}_{\tilde\rho_m}(U): U\in\mathcal{U}_\ell\cap L_m\}<\eta_m.$$
 We can assume that $M_m$ is nondecreasing on $m$. For $m\ge 1$, denote by $\Omega_{m,n}:=V_{m, M_mn}$.
From the non-archimedean property of the absolute value on $\mathbb{C}_p$ and the choices of $L_m$ and $M_m$, we have that $\Omega_{m+1,n}\subset\Omega_{m,n}$ for large $m$.

 Now set $\Omega_0:=J\times J$ and $\Omega_n:=\cap_{m=1}^{\infty}\Omega_{m,n}$ for $n\ge 1$. Note that $\Omega_n\not=\emptyset$ since $\Omega_n$ contains $\Delta:=\{(x,x)\in J\times J\}$.  Let us check that $\Omega_n$ satisfies the assumptions in Lemma \ref{lem:Frink}. The assumption $(a)$  trivially holds. For the assumption $(b)$, from the above, we only need to show $\cap_{n=1}^\infty\Omega_n\subset\Delta$. For any $\epsilon>0$, choose $m_0\ge1$ such that $1/2^m<\epsilon$ for any $m\ge m_0$. Pick $(x,y)\in\cap_{n=1}^\infty\Omega_n$. Then for any $n\ge 1$ and for any $m\ge1$, we have $(x,y)\in\Omega_{m,n}$. It follows that $x$ and $y$ are contained in a $1/2^{m_0}$-neighborhood of an element of $\mathcal{U}_{M_{m_0}n}\cap L_{m_0}$. Then for each $n\ge 1$, there exist $x_n$ and $y_n$ contained in an element of $\mathcal{U}_{M_{m_0}n}\cap L_{m_0}$ such that $|x-x_n|<1/2^{m_0}$ and  $|y-y_n|<1/2^{m_0}$. By the expanding property of $\tilde\rho_{m_0}$, we have that, as $n\to\infty$,
$$\tilde\rho_{m_0}(x_n,y_n)\le\sup\{\mathrm{diam}_{\tilde\rho_{m_0}}(U): U\in\mathcal{U}_{M_{m_0}n}\cap L_{m_0}\}\to 0.$$
We conclude that $|x_n-y_n|<1/2^{m_0}$ for sufficiently large $n$ by Corollary \ref{cor:fine}. It follows that
$$|x-y|\le\max\{|x-x_n|, |x_n-y_n|, |y_n-y|\}<1/2^{m_0}<\epsilon,$$
Hence $x=y$. Thus the assumption $(b)$ holds.

Let us verify the assumption $(c)$. By the non-archimedean property (Corollary \ref{cor:non}), for any $m\ge 1$ and for any $(x,y)\in\Omega_1\circ\Omega_1\circ\Omega_1$, there exists $U_1\in \mathcal{U}_{M_m}\cap L_m$ 
such that $x$ and $y$ are contained in a $1/2^m$-neighborhood of $U_1$. Then by the choice of $M_m$  with, there exists $U\in\mathcal{U}_0\cap L_m$ such that $U_1\subset U$. Hence  $x$ and $y$ are contained in a $1/2^m$-neighborhood of $U$. So $\Omega_1\circ\Omega_1\circ\Omega_1\subset\Omega_0$. Generally,  for any $(x,y)\in\Omega_n\circ\Omega_n\circ\Omega_n$, there exists $U'_n\in\mathcal{U}_{M_mn}\cap L_m$ 
 such that $x$ and $y$ are contained in a $1/2^m$-neighborhood of $U'_n$. Note that there exists $U'_1\in\mathcal{U}_{M_m}\cap L_m$ such that $U'_n\subset\phi^{-M_m(n-1)}(U'_1)$. Since by the choice of $M_m$, there exists  $U'\in\mathcal{U}_0\cap L_m$ such that $U'_1\subset U'$, it follows that $U'_n$ is contained in $\phi^{-M_m(n-1)}(U')$. Denote by $U'_{n-1}$ the disk-component of  $\phi^{-M_m(n-1)}(U')$ containing $U'_{n}$. We have that $x$ and $y$ are contained in a $1/2^m$-neighborhood of $U'_{n-1}$. The assumption $(c)$ follows.

Then by Lemma \ref{lem:Frink}, there exists a metric $\tilde\rho_\infty$ on $J$, compatible with the topology, such that
$$\Omega_n\subset\{(x,y)\in J\times J: \tilde\rho_\infty(x,y)<2^{-n}\}\subset\Omega_{n-1}$$
for any $n\ge 1$. Since $\Omega_{m+1,n}\subset\Omega_{m,n}$, there exists an $m'\ge 1$ such that for any $m\ge m'$ and for any $x$ and $y$ contained in $\Omega_{m,n}$, we have $\tilde\rho_\infty(x,y)<2^{-n}$. Thus for any $U\in\mathcal{U}_{M_{m'}\lfloor n/M_{m'}\rfloor}$, we have
$$\mathrm{diam}_{\tilde\rho_\infty}(U)<2^{-\lfloor n/M_{m'}\rfloor}\le 2\cdot\left(2^{-1/M_{m'}}\right)^n.$$
This completes the proof.
\end{proof}

\begin{corollary}\label{cor:Holder}
Let $\phi$ and $\{L_m\}_{m\ge 1}$ be as above. For each  sufficiently large $m\ge 1$, the identity map $(J_{L_m}(\phi),\tilde\rho_m)\to(J,\tilde\rho_\infty)$ is H\"older continuous.
\end{corollary}
\begin{proof}
Let $M_m$ be as in the proof of Proposition \ref{prop:con}. Pick $x$ and $y$ in $L_m$. Let $n\ge1$ be the smallest integer such that both $x$ and $y$ are contained in $U$ for some $U\in\mathcal{U}_{M_m(n-1)}$. Then by Proposition \ref{prop:con}, we have  $\tilde\rho_\infty(x,y)<2^{-n}$.  Note that
$$\tilde\rho_m(x,y)>C_m\tilde\lambda_m^{-M_m(n-1)},$$
where $1<\tilde\lambda_m<+\infty$ is an upper bound $\tilde\rho_m(\phi(z))|\phi'(z)|/\tilde\rho_m(z)$ for $z\in J_{L_m}(\phi)$ and $C_m:=\min\{\mathrm{diam}_{\tilde\rho_m}(U): U\in\mathcal{U}_0\cap L_m\}$. The desired  H\"older continuity follows.
\end{proof}

\subsection{Semiconjugacy via coding}
Let $\Sigma=\{1,2,\dots,d\}^\mathbb{N}$  be the space of infinite sequences of $d:=\deg\phi$ symbols associated with the natural metric such that $\mathrm{dist}_\Sigma(\alpha,\alpha')=2^{-\min\{j\ge 1: \alpha_j\not=\alpha'_j\}}$ for  $\alpha=(\alpha_j)_{j\ge 1}$ and $\alpha'=(\alpha'_j)_{j\ge 1}$ in $\Sigma$. Let $\sigma:\Sigma\to\Sigma$ be the left shift. In this subsection, we show that $(\Sigma,\sigma)$ and $(J,\phi)$ are semiconjugate. We employ the idea in \cite[Proposition 2.16]{Das21} which is based on the ``geometric coding tree" in \cite[Section 2]{Przytycki85}. Due to the totally disconnectivity of $\mathbb{C}_p$, in the coding we use shortest distance rather than lifting curves as in  \cite[Proposition 2.16]{Das21}.

\begin{proposition}\label{prop:coding}
Fix the notation as above. Then there exists a H\"older continuous and surjective semiconjugacy $h:\Sigma\to J$ such that $h\circ\sigma=\phi\circ h$.
\end{proposition}
\begin{proof}
Pick an initial point $w\in J\setminus P(\phi)$. Then $w$ has $d$ distinct preimages and all of them are contained in $J$. We label them by $w_1,w_2,\dots, w_d$. For any $\alpha:=(i_1,i_2,\dots)\in\Sigma$, we set $z_0(\alpha)=w_{i_1}$. Moreover, we pick any $z_1((1,i_2,\dots))\in\phi^{-1}(w_{i_2})$ such that 
$$|z_1((1,i_2,\dots))-z_0((1,i_2,\dots))|=|z_1((1,i_2,\dots))-w_1|=\mathrm{dist}(z_0(\alpha), \phi^{-1}(w_{i_2})).$$
Now suppose we have $z_1((j,i_2,\dots))\in\phi^{-1}(w_{i_2})$ for $1\le j\le i_1-1$, we set $z_1(\alpha)\in\phi^{-1}(w_{i_2})$ such that $|z_1(\alpha)-z_0(\alpha)|=\mathrm{dist}(z_0(\alpha), \phi^{-1}(w_{i_2}))$. If there are more than one such $z_1(\alpha)$, we pick one of them that is distinct from $z_1((j,i_2,\dots))$ for all $(j,i_2,\dots)\in\Sigma$ with $j\le i_1-1$. Observe that $\phi(z_1(\alpha))=z_0(\sigma(\alpha))$.

 Inductively, suppose that for some $n\ge 2$ and any $\alpha\in\Sigma$, we have $z_{n-1}(\alpha)$ such that $|z_{n-1}(\alpha)-z_{n-2}(\alpha)|=\mathrm{dist}(z_{n-2}(\alpha), \phi^{-(n-1)}(w_{i_n}))$ and  $\phi(z_{n-1}(\alpha))=z_{n-2}(\sigma(\alpha))$. We then set $z_n(\alpha)$ as following. Pick any $z_n(\alpha)\in\phi^{-1}(z_{n-1}(\sigma(\alpha)))$ such that 
 $$|z_n(\alpha)-z_{n-1}(\alpha)|=\mathrm{dist}(z_{n-1}(\alpha), \phi^{-n}(w_{i_{n+1}})).$$
  and 
  $$z_n(\alpha)\not=z_n((j_1,j_2,\dots, j_{n-1},i_{n},\dots))$$ 
  for any $(j_1,j_2,\dots, j_{n-1},i_{n},\dots)\in\Sigma$ with $1\le j_\ell\le d$ for $1\le \ell\le n-2$ and $j_{n-1}\le i_{n-1}-1$.

Now let us show that $\lim\limits_{n\to\infty}z_n(\alpha)$ exists. Denote by $R:=\min\{\mathrm{diam}_{\tilde\rho_\infty}(U): U\in\mathcal{U}_0\}$. Let $0<\Lambda<1$ and $C>0$ be as in Definition \ref{def:con} for the metric $\tilde\rho_\infty$. Then there exists $N_1\ge 1$ such that for any $n\ge N_1$ we have $C\Lambda^{n}<R$. Moreover, by \cite[Corollary 5.21]{Benedetto19} and \cite[Proof of Theorem 4.2.5 (ii)]{Beardon91}, there is
$N_2\ge 1$ such that for any $n\ge N_2$ we have $J\subset\phi^n(U)$ for any $U\in\mathcal{U}_0$, since $U\cap J\not=\emptyset$. Set $N:=\max\{N_1,N_2\}$. It follows that $z_N(\alpha)$ and $z_{N+1}(\alpha)$ are contained in a same $U_0\in\mathcal{U}_0$. Moreover, $z_{n+j}(\alpha)$ and $z_{n+j+1}(\alpha)$ are contained in a same disk component of $\phi^{n+j-N}(U_0)$ for $n\ge N$ and $j\ge 0$. By Proposition \ref{prop:con}, we conclude that for $n\ge N$ and $j\ge 0$,
\begin{align}\label{equ:n}
\tilde\rho_\infty(z_{n+j}(\alpha), z_{n+j+1}(\alpha))<C\Lambda^{n+j-N}.
\end{align}
By the compactness of $J$, the limit $\lim\limits_{n\to\infty}z_n(\alpha)$ exists. We set this limit by $h(\alpha)$ and obtain a well-defined map $h:\Sigma\to J$. Moreover, $\phi\circ h=h\circ \sigma$ since by the construction, $\phi(z_n(\alpha))=z_{n-1}(\sigma(\alpha))$ for any $n\ge 2$ and $\alpha\in\Sigma$.

Now let us show $h$ is H\"older continuous. Consider two elements $\alpha=(i_j)_{j\ge 1}$ and $\alpha'=(i'_j)_{j\ge 1}$ in $\Sigma$ with $\mathrm{dist}_\Sigma(\alpha,\alpha')<1/2$. Let $j_0:=\min\{j\ge 1: i_j\not=i'_j\}$. Then $j_0\ge 2$. It follows that $z_j(\alpha)=z_j(\alpha')$ for $1\le j\le j_0-1$. Thus
$$\tilde\rho_\infty(h(\alpha),h(\alpha'))\le \tilde\rho_\infty(h(\alpha),z_{j_0-1}(\alpha))+\tilde\rho_\infty(z_{j_0-1}(\alpha'), h(\alpha'))$$
If $j_0\ge N+1$, by \eqref{equ:n}, we have
  \begin{multline*}
  \max\left\{\tilde\rho_\infty(h(\alpha),z_{j_0-1}(\alpha)),\tilde\rho_\infty(z_{j_0-1}(\alpha'), h(\alpha'))\right\}
  <\sum_{n=j_0-1}^\infty C\Lambda^{n-N}=\frac{C}{1-\Lambda}\Lambda^{j_0-1-N}.
   \end{multline*}
 It follows that
 \begin{multline*}
 \tilde\rho_\infty(h(\alpha),h(\alpha'))
 \le\frac{2C}{(1-\lambda)\Lambda^{1+N}}2^{-j_0(-\log\Lambda/\log 2)}
 =\frac{2C}{(1-\Lambda)\Lambda^{1+N}}\mathrm{dist}_\Sigma(\alpha,\alpha')^{-\log\Lambda/\log 2}.
 \end{multline*}
 Thus the H\"older continuity of $h$ follows.

 We now show that $h$ is surjective. Pick $x\in J$. For $k\ge 1$, consider the $\tilde\rho_\infty$-disk
 $$D_k:=\{y\in\mathbb{C}_p: \tilde\rho_\infty(x,y)<1/k\}.$$
 Since  $D_k\cap J\not=\emptyset$, there exists $n_k\ge1$ such that $J\subset\phi^{n_k}(D_k)$. It follows that there exists $w_k\in D_k$ such that $\phi^{n_k}(w_k)$ is the initial point $w$. Thus there exists $\alpha_k\in\Sigma$ such that $z_{n_k}(\alpha_k)=w_k$. Since $w_k\to x$, we have $z_{n_k}(\alpha_k)\to x$ as $k\to\infty$. By the compactness of $\Sigma$, passing to subsequence if necessary, we obtain the limit of $\alpha_k$, denoted by $\alpha_\infty\in\Sigma$. Thus for each $m\ge 1$, there exists $k\ge1$ with $n_k\ge m$ such that $\mathrm{dist}_\Sigma(\alpha_\infty,\alpha_k)<2^{-m}$. It follows from \eqref{equ:n} that
 $$\tilde\rho_\infty(z_m(\alpha_\infty), z_{n_k}(\alpha_k))=\tilde\rho_\infty(z_m(\alpha_k), z_{n_k}(\alpha_k))<C'\Lambda^m$$
 for some $C'>0$. Thus
 $$\lim_{m\to\infty}z_m(\alpha_\infty)=\lim_{k\to\infty}z_{n_k}(\alpha_k)=\lim_{k\to\infty}w_k=x.$$
 It follows that $h(\alpha_\infty)=x$. This completes the proof.
 \end{proof}

 \begin{remark}
 In the case that $\phi\in K(z)$ is subhyperbolic but not hyperbolic, due to the presence of Julia critical points, the semiconjugacy in Proposition \ref{prop:coding} can not be upgraded to a conjugacy in general, even for the restriction of Julia dynamics in a finite extension of $K$, see \cite[Theorem 1.4]{Fan21}. While in the case that $\phi\in K(z)$ is hyperbolic, such an upgrade is possible, see \cite[Theorem 3.1]{Kiwi06} for the case that $\phi$ is a hyperbolic polynomial in shift locus.
 \end{remark}

\subsection{Proof of Theorem \ref{thm:CLT}}\label{sec:proof2}
Once Propositions \ref{prop:con} and \ref{prop:coding} are established, we can obtain Theorem \ref{thm:CLT} by a verbatim argument as in \cite{Das21}. For the completeness, we sketch the proof as following. We begin with the topological pressure. Under assumption \ref{assume1}, the Julia set $J$ is a compact metric space. Denote by $\mathcal{M}(\phi)$ the $\phi$-invariant measure probability measures on $J$, and let $h_\mu(\phi)$ be the metric entropy of $\phi$ with respect to $\mu\in\mathcal{M}(\phi)$. Then for a continuous function $f:J\to\mathbb{R}$, \emph{the topological pressure }of $\phi$ with respect to $f$ is defined as
$$\mathcal{P}_{top}(f):=\sup_{\mu\in\mathcal{M}(\phi)}\left\{h_\mu(\phi)+\int_Jfd\mu\right\}.$$
A measure $\mu\in\mathcal{M}(\phi)$ is an \emph{equilibrium state} for $f$ if it realizes the above supremum.

Recall from Proposition \ref{prop:coding} that $h:\Sigma\to J$ is the semiconjugacy. The following result can be obtained  as in \cite[Sections 3 and 4]{Das21}
\begin{lemma}\label{lem:same}
Fix the notation as before. Then the following hold.
\begin{enumerate}[{\rm (1)}]
\item {(No entropy drop)} Let $\nu$ be a $\sigma$-invariant probability measure on $\Sigma$, and let $\mu:=h_\ast\nu$ be the pushforward measure on $J$. Then $h_\nu(\sigma)=h_\mu(\phi)$.
\item {(Lift of invariant measures)} For any $\phi$-invariant probability measure $\mu$ on $J$, there exists a $\sigma$-invariant probability measure $\nu$ on $\Sigma$ such that $\mu=h_\ast\nu$.
\item {(Pushforward equilibrium states)} Let $f:(J,\tilde\rho_\infty)\to\mathbb{R}$ be a H\"older continuous function. Then
\begin{enumerate}[{\rm (a)}]
\item if $\nu$ is an equilibrium state for $f\circ h$, then $\mu:=h_\ast\nu$ is an equilibrium state for $f$; and
\item if $\mu$ is an equilibrium state for $f$, then there exists a measure $\nu$ that is an equilibrium state for $f\circ h$ and satisfies $h_\ast\nu=\mu$.
\end{enumerate}
\end{enumerate}
\end{lemma}

Then existence and uniqueness of $\mu_f$ and statements (1)-(4) in Theorem \ref{thm:CLT} follow from Lemma \ref{lem:same} and  the corresponding statements in the dynamical system $(\Sigma,\sigma)$, see \cite[Lemma 4.4 and the proof of Theorem 1.1 (1)-(5) in p. 180]{Das21}. This is based on \cite[Theorem 5.4.9 and Theorem 5.7.1]{Przytycki10} and \cite[Theorem 1.1]{Denker01}

Moreover, since for any finite extension $L$ of $K$, the map $\phi$ uniformly expands disks in $L$ with respect to $\rho_L:=\rho_{L,1}$ defined in \eqref{rho},
by Corollary \ref{cor:Holder}, we can obtain statements (5)-(6) in Theorem \ref{thm:CLT}, applying a similar argument as in \cite[Section 6]{Das21}. \hfill\qedsymbol

\bibliographystyle{acm}

\end{document}